\documentclass[12pt]{amsart}

\usepackage{times}      
\usepackage{amssymb}    
\usepackage{amsmath}    
\usepackage{amsthm}
\usepackage{tikz-cd}
\usepackage{placeins}
\usepackage{hyperref}

\usepackage{url}
\makeatletter
\def\url@leostyle{%
  \@ifundefined{selectfont}{\def\UrlFont{\sf}}{\def\UrlFont{\scriptsize\ttfamily}}}
\makeatother
\newtheorem{theorem}{Theorem}[section]
\newtheorem{corollary}[theorem]{Corollary} 
\newtheorem{lemma}[theorem]{Lemma} 
\newtheorem{proposition}[theorem]{Proposition} 

\newtheorem*{thma}{Theorem A}
\newtheorem*{thmb}{Theorem B}
\newtheorem*{corc}{Corollary C}

\theoremstyle{definition}
\newtheorem{definition}[theorem]{Definition} 
\newtheorem*{acknowledge}{Acknowledgements}

\newtheorem{example}[theorem]{Example} 

\def\nnn{\mathbb{N}}
\def\rrr{\mathbb{R}}
\def\ccc{\mathbb{C}}
\def\zzz{\mathbb{Z}}

\def\o{\mathcal{O}^{\cdot, D, \cdot}_{k, \cdot, v} (n)}
\def\a{\mathrm{Alex}^{\cdot, D, \cdot}_{k, \cdot, v} (n)}
\def\m{\mathcal{M}^{\cdot, D, \cdot}_{k, \cdot, v} (n)}
\def\on{\mathrm{O}(n)}
\def\meq{\mathcal{M}^{\mathrm{c}}_{\mathrm{eq}}}
\def\ooo{\Omega}

\def\ali{^{\alpha}_i}
\def\pal{p^{\alpha}}
\def\pali{p^{\alpha}_i}
\def\ual{U^{\alpha}}
\def\uali{U^{\alpha}_i}

\def\cal{\tilde{U}^{\alpha}}
\def\cali{\tilde{U}^{\alpha}_i}

\def\id{\mathrm{id}_G}
\def\dist{\mathrm{dist}}
\def\vol{\mathrm{vol}}
\def\co{\colon\thinspace}
\def\diam{\mathrm{diam}}
\def\embed{\hookrightarrow}

\def\st{\mathrel{}\middle|\mathrel{}}

\begin{document}

\title{Equivariant Alexandrov geometry and orbifold finiteness}

\author{John Harvey}
\address{Mathematisches Institut, Universit\"{a}t M\"{u}nster, Einsteinstr. 62, 48149 M\"{u}nster, Germany. }
\email{harveyj@uni-muenster.de}


\date{May 6, 2015.}
\thanks{This paper has been published by the Journal of Geometric Analysis, with DOI \href{http://dx.doi.org/10.1007/s12220-015-9614-6}{10.1007/s12220-015-9614-6}.}

\begin{abstract}
Let a compact Lie group act isometrically on a non-collapsing sequence of compact Alexandrov spaces with fixed dimension and uniform lower curvature and upper diameter bounds. If the sequence of actions is equicontinuous and converges in the equivariant Gromov--Hausdorff topology, then the limit space is equivariantly homeomorphic to spaces in the tail of the sequence.

As a consequence, the class of Riemannian orbifolds of dimension $n$ defined by a lower bound on the sectional curvature and the volume and an upper bound on the diameter has only finitely many members up to orbifold homeomorphism. Furthermore, any class of isospectral Riemannian orbifolds with a lower bound on the sectional curvature is finite up to orbifold homeomorphism.
\end{abstract}

\maketitle

\section{Introduction}

The Gromov--Hausdorff topology on the set of all compact metric spaces has been widely studied since its introduction by Gromov in 1981 \cite{gromov}. Consideration of this topology led naturally to the definition of new classes of metric spaces of geometric interest. The present work considers Alexandrov spaces. 

An Alexandrov space has a lower curvature bound which generalizes the lower sectional curvature bound on a Riemannian manifold. These spaces arise naturally as limits of sequences of Riemannian manifolds with a uniform lower sectional curvature bound. 

One of the deepest results in Alexandrov geometry is Perelman's Stability Theorem \cite{perstab}, which states that if a sequence of compact Alexandrov spaces has a uniform lower curvature bound, and neither grows unboundedly in terms of its diameter nor collapses in terms of its dimension, its topological type does not change on passage to the limit. 

This result is almost omnipresent in Alexandrov geometry. One may construct the tangent cone of an Alexandrov space at a point $p$ by taking the limit of the space under rescaling around $p$. The Stability Theorem  shows that the space is locally homeomorphic to its tangent cone, and therefore, at least topologically, its singularities are very controlled. 

It is desirable to obtain an analogous convergence result in the equivariant setting. Here the appropriate topology is Fukaya's equivariant Gromov--Hausdorff topology \cite{fukaya}.

In this vein, Searle and the author showed that an isometric action on an Alexandrov space is locally determined by the isotropy action at the point \cite{HS}. The main theorem of the present work gives a sufficient condition for a convergent sequence of $G$--actions on Alexandrov spaces with a uniform lower curvature bound to be stable, in the sense that the limiting action is equivariantly homeomorphic to those in the tail of some subsequence.

\begin{thma}\label{t:equistab}Let $G$ be a compact Lie group and let $X_i$ be a sequence of compact Alexandrov spaces of fixed dimension $n$, with curvature bounded below by $k$ and diameter bounded above by $D$, each with an effective isometric action of $G$. Suppose that $(X_i, G)$ converges in the equivariant Gromov--Hausdorff topology to $(X,\Gamma)$, where $X$ is also of dimension $n$. Suppose further that the sequence of actions is equicontinuous. 
	
Then the groups $G$ and $\Gamma$ are isomorphic and, for large $i$, the spaces $X_{i}$ are equivariantly homeomorphic to $X$. 

Furthermore, if $\theta_i \co X/G \to X_i / G$ are homeomorphisms which witness the Gromov--Hausdorff convergence of the orbit spaces, then there is a subsequence for which the equivariant homeomorphisms $X \to X_i$ can be chosen so that they descend to $\theta_i$.
\end{thma}

The theorem can also be stated as follows: In the space of Alexandrov $G$--spaces of dimension $n$ with curvature bounded below by $k$ with the equivariant Gromov--Hausdorff topology, every point has a neighborhood consisting only of Alexandrov spaces to which it is equivariantly homeomorphic.

Placed in this form, the original Stability Theorem is also a major contribution towards the question of how the geometry of a Riemannian manifold controls its topology.
The problem of finding geometric constraints to define a  class of manifolds which is finite up to homotopy, homeomorphism or diffeomorphism has a long history.

If the class is defined by bounds on sectional curvature, diameter and volume, then a convenient notation is to write $\mathcal{M}^{K, D, V}_{k, d, v} (n)$.
This represents the class of all Riemannian manifolds $(M,g)$ with $k \leq \sec_g \leq K$, $d \leq \diam (M) \leq D$ and $v \leq \vol (M) \leq V$. 
Where a value is replaced with ``$\cdot$'' the condition is understood to be deleted.

The first such result is that of Weinstein, who showed that, for $\delta > 0$, $\mathcal{M}^{1, \cdot, \cdot}_{\delta, \cdot, \cdot} (2n)$, the class of unformly pinched positively curved manifolds of even dimension, has only finitely many members up to homotopy \cite{weinstein}.
Shortly after this, Cheeger showed that, for $n \neq 4$, $\mathcal{M}^{K, D, \cdot}_{k, \cdot, v} (n)$ has finitely many simply connected members up to diffeomorphism  \cite{cheeger}. In \cite{peters}, Peters shows that this result still holds for $n=4$, as well as removing the hypothesis of simple connectivity.

Grove and Petersen removed the upper bound on sectional curvature, and obtained finiteness of $\m$ up to homotopy \cite{gphomotopy}.
Shortly afterwards, in collaboration with Wu, this result was improved to show finiteness up to homeomorphism for $n \geq 4$ \cite{gpw}. 
As long as the dimension is not four, the work of Kirby and Siebenmann \cite{ks} implies finiteness up to diffeomorphism.

Perelman's Stability Theorem \cite{perstab} showed that $\a$, the corresponding class of Alexandrov spaces, is finite up to homeomorphism. \emph{A fortiori}, this generalizes the Grove--Petersen--Wu finiteness result to all dimensions.

Theorem B uses the equivariant version of the Stability Theorem to generalize the homeomorphism finiteness result of Grove, Petersen and Wu to the area of Riemannian orbifolds. 
An orbifold is a mild generalization of a manifold, and, to give just a few examples, the concept has found applications in Thurston's work on the Geometrization Conjecture \cite{thurston}, the construction of a new positively curved manifold by Dearricott \cite{Dear} and Grove--Verdiani--Ziller \cite{gvz}, and string theory, such as Dixon, Harvey, Vafa and Witten's conformal field theory built on a quotient of a torus \cite{dhvw}.
The same convenient notation can be used for orbifolds, here replacing $\mathcal{M}$ with $\mathcal{O}$.

The first finiteness result for orbifolds is that of Fukaya \cite{fukaya}, who generalized the result of Cheeger, showing that a subclass of $\mathcal{O}^{K, D, \cdot}_{k, \cdot, v} (n)$ is finite up to orbifold diffeomorphism. 
Fukaya used a much more restrictive definition of orbifold, considering only the orbit spaces of global actions by finite groups on Riemannian manifolds. 
This corresponds to what Thurston called a ``good'' orbifold \cite{thurston}.

Working in dimension two, Proctor and Stanhope showed that $\mathcal{O}^{\cdot, D, \cdot}_{k, \cdot, v} (2)$ is finite up to orbifold diffeomorphism \cite{procstan}, providing a first generalization of the result of Grove, Petersen and Wu.
The homeomorphism finiteness result was then shown in all dimensions by Proctor, provided the orbifold has only isolated singularities \cite{proc}. 
Here that assumption is removed, completing the generalization of Grove--Petersen--Wu's homeomorphism finiteness.

\begin{thmb}
	For any $k, D, v, n$, the class $\o$ has only finitely many members up to orbifold homeomorphism.
\end{thmb}

By Weyl's asymptotic formula, which Farsi has shown is valid for orbifolds \cite{farsi}, a Laplace isospectral class of orbifolds has fixed volume and dimension. Stanhope has shown that, in the presence of a lower bound on Ricci curvature, such a class has a uniform upper bound on its diameter \cite{stan}, and so, just as in \cite{proc}, the following corollary is clear.

\begin{corc}
	Any class of Laplace isospectral orbifolds with a uniform lower bound on its sectional curvature has only finitely many members up to orbifold homeomorphism.
\end{corc}

This generalizes the similar result of Brooks, Perry and Petersen for Laplace isospectral manifolds \cite{bpp}. While one cannot hear the shape of an orbifold, one can, at least in the presence of a lower sectional curvature bound, know that there are only finitely many possibilities.

\begin{acknowledge}This research was carried out as part of the author's dissertation project at the University of Notre Dame, with the ever-helpful advice of Karsten Grove. During that time, the author was supported in part by a grant from the U.S. National Science Foundation.
	
The author is grateful to Vitali Kapovitch and Curtis Pro for interesting and helpful conversations on this subject, and to Karsten Grove for pointing out the possibility of using \cite{gkr} to prove Proposition \ref{p:equicontinuity}. \end{acknowledge}

\section{Preliminaries}

\subsection{Gromov--Hausdorff topologies}

A general approach for proving finiteness results such as Theorem B \cite{gpw,perstab,proc} is to combine a compactness or precompactness result for the class under consideration with a stability result. 
A particularly useful topology (in fact, a metric) on the set of isometry classes of compact metric spaces was proposed by Gromov \cite{gromov}. 
Gromov's metric generalizes the Hausdorff metric on the closed subsets of a compact metric space.

\begin{definition}Let $(X,d_X)$ and $(Y,d_Y)$ be metric spaces. A function $f \co X \to Y$ (not necessarily continuous) is called an Gromov--Hausdorff $\epsilon$--approximation if, for all $p,q \in X$, $\left| d_X(p,q) - d_Y(f(p),f(q)) \right| \leq \epsilon$ and an $\epsilon$--neighborhood of the image of $f$ covers all of $Y$.\end{definition}

\begin{definition}The \emph{Gromov--Hausdorff distance} between two compact metric spaces $(X, d_X)$ and $(Y, d_Y)$ is the infimum of the set of all $\epsilon$ such that there are Gromov--Hausdorff $\epsilon$--approximations $X \to Y$ and $Y \to X$.\end{definition}

The equivariant Gromov--Hausdorff topology was first defined by Fukaya \cite{fukaya}, and achieved its final form some years later in his work with Yamaguchi \cite{fyannals}. 
Consider the set of ordered pairs $(M, \Gamma)$ where $M$ is a compact metric space and $\Gamma$ is a closed group of isometries of $M$. Say that two pairs are equivalent if they are equivariantly isometric up to an automorphism of the group.
Let $\meq$ be the set of equivalence classes of such pairs.

\begin{definition}Let $(X,\Gamma),(Y, \Lambda) \in \meq$. An \emph{equivariant Gromov--Hausdorff $\epsilon$--approximation} is a triple $(f,\phi,\psi)$ of functions $f \co X \to Y$, $\phi \co \Gamma \to \Lambda$ and $\psi \co \Lambda \to \Gamma$ such that
	\begin{enumerate}
		\item $f$ is an Gromov--Hausdorff $\epsilon$--approximation;
		\item if $\gamma \in \Gamma, x \in X$, then $\dist(f(\gamma x) , \phi(\gamma) f(x) ) < \epsilon$; and
		\item if $\lambda \in \Lambda, x \in X$, then $\dist(f(\psi(\lambda) x) , \lambda f(x) ) < \epsilon$.
	\end{enumerate}
\end{definition}

Note that these functions need not be morphisms from the relevant category. The equivariant Gromov--Hausdorff distance is defined from these approximations just as with the standard Gromov--Hausdorff distance.

An alternative definition was provided by Paulin (attributed by him to Bonahon) \cite{paulin}. This definition requires the same group to act on both spaces. A different Gromov--Hausdorff approximation is used for each finite subgroup, and  that approximation must be exactly equivariant with respect to the action of the subgroup. Under this definition two spaces might be considered to be separated by a positive distance if they differ only by an automorphism of the group.

By \cite[Proposition 3.6]{fyannals}, given a sequence in $\meq$, if the sequence of underlying metric spaces converges in the Gromov--Hausdorff topology to a compact metric space then there is a subsequence which converges in the equivariant Gromov--Hausdorff topology.

By \cite[Theorem 2.1]{fukaya}, the sequence of orbit spaces corresponding to a convergent sequence in $\meq$ must itself converge in the usual Gromov--Hausdorff topology.

The following two examples demonstrate the types of convergence that can occur without the hypotheses of Theorem A. In the first case, the group is not fixed. In the second example, the group has been fixed but its actions are not equicontinuous.

\begin{example}Let $\zzz_p$, the cyclic group of order $p$, act freely on $S^3$ with orbit space $S^3 / \zzz_p \cong L_{p,1}$. Then, as $p \to \infty$, the limit action is that of a circle. The lens spaces collapse to a limit orbit space homeomorphic to $\ccc \mathrm{P}^1$.\end{example}

\begin{example}Let $T^2$ act isometrically on the round sphere $S^3$. This torus has two distinguished circle subgroups which act so as to give a disk for orbit space. Consider the circle subgroup $S^1_p$ of $T^2$ which winds around the first of these subgroups $p$ times and the second once. The orbit space of this circle action is the so-called ``weighted'' projective space $\ccc \mathrm{P}^1_{p,1}$. The limit action on this occasion is that of the full $T^2$. The weighted projective spaces collapse to a limit orbit space homeomorphic to an interval.\end{example}

Convergence of non-compact spaces can also be defined by adding a basepoint, provided that the closed metric balls around the basepoint are compact. Such sequences are said to converge if, for every $r > 0$, the closed metric balls of radius $r$ around the basepoint converge. Where equivariant  convergence of non-compact spaces is considered in the present work, the basepoint will always be fixed by the group. In this case, convergence also reduces to the convergence of closed balls.

\subsection{Basics of Alexandrov geometry}

Certain curvature conditions define precompact subsets of the set of all compact metric spaces. 
For example, Gromov showed that the class of all Riemannian manifolds of dimension $n$, with diameter less than $D$, and with Ricci curvature greater than $(n-1)k$ is precompact \cite{gromov}. 
Strengthening the curvature condition to require a lower bound on the sectional curvature provides much more structure on the limit spaces, and it is in this context that Alexandrov geometry was first studied.

It is possible to show that, for a Riemannian manifold, the condition that sectional curvature be $\geq k$ can be expressed as a triangle-comparison condition.
Grove and Petersen showed \cite{gpjdg} that the closure of $\m$ is contained within the class of all complete length metric spaces satisfying this triangle-comparison condition.
It is natural, then, to study this class in its own right.

\begin{definition}
An \textit{Alexandrov space} of finite dimension $n\geq 1$ is a locally complete, locally compact, connected  length space, with a lower curvature bound in the triangle-comparison sense. By convention, a $0$--dimensional Alexandrov space is either a one-point or a two-point space.
\end{definition}

Many fundamental results in this area were proved by Burago, Gromov and Perelman \cite{bgp}, and this paper is a good general reference for the subject. 
They showed that the class of all Alexandrov spaces with curvature bounded below by $k$ is closed under passing to Gromov--Hausdorff limits, and under quotients by isometric group actions. Given a sequence of spaces with a uniform lower curvature bound and fixed dimension $n$, the limit space has dimension at most $n$. 

Let $X$ be an Alexandrov space, and let $p \in X$. 
Then, also by \cite{bgp}, there is a uniquely defined tangent cone at $p$, $T_p X$, which can be obtained as a limit object by rescaling $X$ around $p$.
$T_p X$ is itself an Alexandrov space, with curvature $\geq 0$. 

The most important singularities of an Alexandrov space are its extremal subsets, introducted by Perelman and Petrunin \cite{perpet}.
The distance functions in an Alexandrov space have well-defined gradients, and it is possible to flow along these gradients. 
The gradient flow gives a natural way to understand an extremal subset.

\begin{definition}Let $X$ be an Alexandrov space. A subset $E \subset X$ is extremal if, for every $p \in X$, the flow along the gradient of $\dist(p,\cdot)$ preserves $E$.\end{definition}

Trivial examples of extremal sets are the empty set, and the entire space $X$. Any point having a space of directions with diameter $\leq \pi/2$ is extremal, as is the boundary of an Alexandrov space. The extremal subsets stratify the space into manifolds, and informally they are usefully thought of as strata with small normal spaces. Of greatest interest for the topic under discussion is the following result \cite{perpet}.

\begin{proposition}\label{p:extremalisotropy}Let $X$ be an Alexandrov space, and let $G$ be a compact Lie group acting on $X$ by isometries. Let $X^H$ be the closure of the set of points in the orbit space $X/G$ which are the image of points with isotropy $H$. Then $X^H$ is an extremal subset of $X/G$.\end{proposition}

Extremal sets survive the passage to Gromov--Hausdorff limits, and so for any extremal set $E$, and any point $p \in E$, there is a well defined tangent subcone $T_p E \subset T_p X$ which is also extremal. Conversely, if $E$ is a closed subset of $X$ such that $T_p E$ is extremal for each $p \in E$, then $E$ is an extremal subset.

\subsection{The Stability Theorem}

A crucial advance in the understanding of Alexandrov spaces was made by Perelman with his proof of the Stability Theorem \cite{perstab}. 
The author recommends the treatment by Kapovitch \cite{kapstab} for those who wish to learn more about this deep result. 

The statement of the theorem given here is a relative version of Perelman's original theorem. 
It was proved by Kapovitch for the case where only one extremal subset is under consideration, but as was pointed out by Searle and the author \cite{HS}, it is in fact true in greater generality. 

\begin{theorem}[Stability Theorem \cite{perstab, kapstab,HS}]\label{t:relstability}
 Let $X_i$ be a sequence of compact Alexandrov spaces of dimension $n$ with curvature uniformly bounded from below, converging to a compact Alexandrov space $X$ of the same dimension. 
 Let $\mathcal{E}_i = \{ E^{\alpha}_i \subset X_i \}_{\alpha \in A}$ be a family of extremal sets in $X_i$ indexed by a set $A$,  converging to a family of extremal sets $\mathcal{E}$ in $X$.
 
 Let $o(i) \co \nnn \to (0,\infty)$ be a function with $\lim_{i \to \infty} o(i) = 0$. Let $\theta_i\co  X \to X_i$ be a sequence of $o(i)$--Gromov--Hausdorff approximations. 
 
Then for all large $i$ there exist homeomorphisms $\theta'_i \co  (X, \mathcal{E}) \to (X_i, \mathcal{E}_i)$, $o(i)$--close to $\theta_i$.
 \end{theorem}

This result implies all the previously known finiteness results for manifolds, other than diffeomorphism finiteness in dimension four.
It also has a vital application in Alexandrov geometry.
Consider the construction of the tangent cone to an Alexandrov space by the convergence of the sequence obtained by rescaling the metric around a certain point.
By a non-compact version of Theorem \ref{t:relstability}, the local structure of the space is controlled by the tangent cone.

\begin{corollary}Let $X$ be an Alexandrov space, and let $p \in X$. Then for some $r_0 >0$, $B_r(p) \cong T_p X$ for all $r < r_0$. Furthermore, for small enough $r_0$ the homeomorphism can be chosen so that, for every extremal set $E$, $E \cap B_{r}(p)$ is mapped to $T_p E$.\end{corollary}

These small conical neighborhoods are extremely useful in the study of Alexandrov spaces, and so it will be convenient to make the following definition.

\begin{definition}
	An open subset $U$ of an Alexandrov space $X$ is called \emph{cone-like around $p$} if $p \in U$, and there is a homeomorphism $f \co U \to T_p X$ with $f(p)$ being the vertex of the cone and $f(E \cap U) = T_p E$ for each extremal set $E$.
\end{definition}

For the proof of Theorem A, it will also be necessary to require the stability homeomorphisms to behave in a particular manner near a point, or near an orbit of a group action. 

\begin{proposition}\label{p:localstab}Under the assumptions of Theorem \ref{t:relstability}, let $p \in X$ and let $p_i \in X_i$ converge to $p$. Then there is a small $r > 0$ such that for $0 < \delta < r$ and large $i$ the homeomorphisms $\theta'_i$ can be chosen to also respect the distance from $p$ in the annulus around $p$. More precisely, for all $q \in B_r(p) \setminus B_{\delta}(p)$, $\dist(p_i, \theta'_i (q)) = \dist(p,q)$. 
	
If each of the $X_i$ and $X$ admit an isometric action by compact Lie groups $G_i$ and $G$, and these actions form a convergent sequence, then for some subsequence the points $p_i$ and $p$ may be replaced with the orbits $G_i \cdot p_i$ and $G \cdot p$.\end{proposition}

A brief proof will be given now, but reference to \cite{kapstab} is advised for a full understanding of the details.

The proof of the stability theorem is carried out on a local basis. 
The space $X$ is covered by compact sets which are said to be \emph{framed}.

\begin{definition}A compact subset $P$ of an Alexandrov space $X$ is called $k$--framed if $P$ has a finite open cover $U_{\alpha}$ such that there are regular maps $f_{\alpha} \co U_{\alpha} \to \rrr^k$. In other words, $P$ is covered by fiber bundles over subsets of $\rrr^k$.\end{definition}

If a $k$--framed set in $X$ has a lift to $X_i$, then it is possible to use the framing to construct a homeomorphism between the framed sets.
These local homeomorphisms are all glued together to construct the global homeomorphism.
All of these results can be proved in parametrized versions, so that the homeomorphisms respect certain maps.

\begin{proof}[Proof of Proposition \ref{p:localstab}]For a suitable choice of $r$ the function $f(q)=\dist(p,q)$ is regular on $B_r(p) \setminus B_{\delta}(p)$ as well as on $B_r(p_i) \setminus B_{\delta}(p_i)$ for large $i$, depending on $\delta$.
Cover $X$ with framed sets so that for every framed set $P$ which intersects this annulus, $f$ is the first co-ordinate of every framing map $f_{\alpha}$. 
Then the local homeomorphisms between framed sets will all respect $f$ on the annulus.
The gluing of the local homeomorphisms can be carried out to respect $f$ on the annulus as well.

For the case of a group action, the orbit spaces $X_i / G_i$ converge to $X / G$, and the distance functions in the orbit spaces have the necessary regularity property. 
The lifts of these functions to $X_i$ and $X$ are regular over points where they are regular in the orbit space, and so the proof can be applied in this case also.\end{proof}

\section{Equivariant stability}

It is well known that a Gromov--Hausdorff convergent sequence can, after passing to a subsequence, be reduced to a Hausdorff convergent sequence in an enveloping metric space. This provides a more concrete object of study, adding some convenience. This result can be generalized to the equivariant setting, and so the slice theorem holds in the enveloping metric space. The proof of  Theorem A is based on a study of this enveloping space.

First, however, it is necessary to investigate the phenomenon of equicontinuous sequences of actions.
The additional assumption of equicontinuity is needed to establish the existence of a limit.

\begin{definition}\label{d:equicont}Let $X_i$ be a sequence of compact metric spaces, and let $G$ be a compact Lie group which acts by isometries on each of them by a map $\rho_i \co G \times X_i \to X_i$. Then the sequence of $G$--actions will be called \emph{equicontinuous} if, for some fixed metric on $G$, and for every $\epsilon > 0$, there is a $\delta > 0$ such that for every $g \in G, p \in X_i$ and for each $i$, $\rho_{i}^{-1} (B_{\epsilon}(\rho_i(g,p)))$ contains a ball of radius $\delta$ around $(g,p)$ in the product metric.\end{definition}

As will be seen in Lemma \ref{l:ambientconvergence}, equicontinuity implies convergence. 
The converse is a little trickier.
In fact, if the representatives from an equivalence class in $\meq$ are chosen in a particular way, even a constant sequence in $\meq$ might not be equicontinuous. 
For example, consider an action of $T^2$ on a metric space $X$. By changing the group by a sequence of automorphisms $\left(\begin{smallmatrix}1&k\\0&1\end{smallmatrix}\right) \in \mathrm{SL}(2,\zzz)$ with $k \to \infty$ a non-equicontinuous sequence of equivalent actions on $X$ is generated.

The following proposition shows that, at least in the case of Alexandrov spaces with a uniform lower curvature bound, it is always possible to find appropriate automorphisms rendering a convergent sequence equicontinuous.

The proof of this result relies on the center of mass construction from Grove--Petersen \cite{gphomotopy}, which allows for the construction of continuous maps from discrete ones.
In the review of this construction, bear in mind that the Riemannian manifold will be the compact Lie group $G$ with a bi-invariant metric.

Let $(M,g)$ be a complete Riemannian manifold, with $\dim M = n$, $\sec g \geq k$, $\vol (M,g) \geq v$ and $\diam (M,g) \leq D$.

A minimal $\mu$--net for $M$ is defined to be a set of points in $M$ such that the $\mu$--balls cover all of $M$ but the $\frac{\mu}{2}$--balls are disjoint.

It is shown in \cite{gphomotopy} that certain constants $r, R > 0$ and $N \in \nnn$ exist which depend only on $n, k, v$ and $D$, but not on the manifold $M$ itself, so that the following hold:	
\begin{enumerate}
	\item For any minimal $\mu$--net, a ball of radius $\mu$ will have non-empty intersection with at most $N$ of the $\mu$--balls centered on the members of the $\mu$--net.
	$N$ depends only on $n, k$ and $D$.
	
	\item Let $p_1, \ldots p_m \in M$,  and let $\lambda_1, \ldots, \lambda_m > 0$ be weights, so that $\Sigma \lambda_i = 1$. Let $\eta < r(1 + R + \cdots +R^{m-1})^{-1}$. If $\dist(p_i, p_j) < \eta$, $i, j = 1, \ldots , m$, then a center of mass $\mathcal{C}(p_1, \ldots p_m, \lambda_1, \ldots, \lambda_m)$ is defined which depends continuously on the $p_i$ and the $\lambda_i$, is unchanged on dropping any point with weight 0, and satisfies $\dist(\mathcal{C},p_i) < \eta (1 + R + \cdots +R^m)$ for each $i$.
\end{enumerate}

Write $K=1 + R + \cdots +R^{N}$.

\begin{proposition}\label{p:equicontinuity}
Let $G$ be a compact Lie group acting by isometries on a compact Alexandrov space $X$ of dimension $n$ and curvature bounded below by $k$. Suppose that a sequence of Alexandrov spaces $X_i$ with the same dimension $n$ and lower curvature bound $k$ is also acted on isometrically by $G$. If $(X_i,G)$ converges to $(X,G)$ in the equivariant Gromov--Hausdorff topology, then there is always an equicontinuous sequence of spaces equivalent to $(X_i,G)$ in $\meq$.	
\end{proposition}

\begin{proof}

First fix a bi-invariant Riemannian metric $\sigma$ on $G$, and fix constants $N$ and $K$ as above, which are appropriate to the metric $\sigma$.
When a compact Lie group acts on a compact metric space by isometries, the action induces a continuous norm on the group. 
Let $\rho_i, \rho$ be norms on $G$ defined by $\rho_i(g) = \sup_{x \in X_i} \dist (x,gx)$ and $\rho(g) = \sup_{x \in X} \dist (x,gx)$. 
The norms induce distance functions by $d_{\rho} (g,h) = \rho(gh^{-1})$.
These distance functions are continuous with respect to the Riemannian metric.

Choose a sequence $\epsilon_i \to 0$ and $(f_i,\phi_i,\psi_i)$, a triple of functions $f_i \co X \to X_i$, $\phi_i \co G \to G$ and $\psi_i \co G \to G$ such that
\begin{enumerate}
	\item $f_i$ is an Gromov--Hausdorff $\epsilon_i$--approximation;
	\item if $g \in G, x \in X$, then $\dist(f_i(g x) , \phi_i(g) f_i(x) ) < \epsilon_i$; and
	\item if $g \in G, x \in X$, then $\dist(f_i(\psi_i(g) x) , g f_i(x) ) < \epsilon_i$.
\end{enumerate}

\begin{lemma}
	The function $\psi_i$ may be chosen to be continuous.
\end{lemma}

\begin{proof}
	
	Let $\nu_i > 0$ be such that $d_{\sigma}(g,h) < 2 \nu_i \implies d_{\rho_i}(g,h) < \epsilon_i$.
	Let $A_i$ be a minimal $\nu_i$--net in $(G,\sigma)$.
	Let $\eta_i > 0$ converge to 0, but let each $\eta_i$ be large enough that $d_{\rho}(g,h) < 4 \epsilon_i \implies d_{\sigma}(g,h) < \eta_i$, and choose a minimal $\eta_i$--net $B_i \subset (G,\sigma)$.
	
	Define a map $\alpha \co A_i \to B_i$ by mapping $p \in A_i$ to an element of $B_i$ nearest (in the $\sigma$ metric) to $\psi_i(p)$.
	If, for some $p,q \in A_i$, $d_{\sigma}(p,q) < 2 \nu_i$, then $d_{\sigma}(\alpha(p),\alpha(q)) < 3 \eta_i$.
	There is an induced map between the Euclidean spaces $\rrr^{A_i} \to \rrr^{B_i}$, where the coordinate for any $p \in B_i$ is obtained by summing the co-ordinates for each element of $\alpha^{-1}(p)$.
	
	Then a continuous map $\tilde{\psi_i} \co (G,\rho_i) \to (G,\rho)$ may be defined by composing maps $(G,\rho_i) \to \rrr^{A_i} \to \rrr^{B_i} \to (G,\rho)$.
	
	Let $A_i = \left\lbrace p_i^1 , \ldots, p_i^{\ell} \right\rbrace $ and choose smooth functions $f_i^j \co (G, \sigma) \to \rrr^{A_i}$, each having their support in the ball of radius $\nu_i$ around $p_i^j$, with $\Sigma_j f_i^j = 1$ and $f_i^j (p_i^k) \neq 0 \implies j=k$.
	The map $(G,\rho_i) \to \rrr^{A_i}$ is given by $g \mapsto f_i^j(g)$.
	Note that points in the image of this map have at most $N$ non-zero coordinates.
	
	The map from $\rrr^{A_i} \to \rrr^{B_i}$ is that induced by $\alpha$, and the map from $\rrr^{B_i} \to (G,\rho)$ is given by the center of mass construction.
	Note that in the domain points have at most $N$ non-zero coordinates, and the corresponding elements of $B_i$ are at pairwise distance at most $3 \eta_i$.
	
	To verify that $\tilde{\psi_i}$ will serve as part of the equivariant Gromov--Hausdorff approximation, it is enough to show that $d_{\rho}(\psi_i(g),\tilde{\psi_i}(g))$ is uniformly bounded over $G$, and that this bound is converging to zero.
	Let $p_i^1, \ldots, p_i^{N_2}$ be those elements of $A_i$ within $\nu_i$ of $g$ in the $\sigma$ metric. 
	Their images in $B_i$ under $\tilde{\psi_i}$ are then at most $\eta_i$ from $\psi_i(g)$ in the $\sigma$ metric.
	The point $\tilde{\psi_i}(g)$ is obtained from the elements of $B_i$ via the center of mass construction, and so is at most $3 \eta_i K$ from any of them. 
	This gives a global bound of $\eta_i(3K+1)$ for the difference between $\psi_i$ and $\tilde{\psi_i}$.
\end{proof}

Now, by the continuity of the $\rho$ distance function with respect to $\sigma$, it is clear that for large $i$ the (now assumed to be continuous) map $\psi_i$ will be an almost homomorphism in the sense of Grove--Karcher--Ruh \cite{gkr}. 
That is to say, for each $g,h \in G$, $d_{\sigma}(\psi_i(gh)\psi_i(h)^{-1},\psi_i(g)) \leq q$ for a fixed small $q$. 
By \cite[Theorem 4.3]{gkr}, there is then a continuous group homomorphism within $1.36q$ of $\psi_i$, and again by continuity of $\rho$, for large enough $i$ this means $\psi_i$ may be taken to be a homomorphism of Lie groups.

Finally, it is necessary to check that it is in fact an \emph{isomorphism} of Lie groups. Let $H_i$ be the kernel of $\psi_i$. Then $(X,H_i) \to (X,1)$ in $\meq$. It follows that $X/H_i \to X$ in the Gromov--Hausdorff topology. However, since dimension cannot increase in the limit \cite{bgp}, and since the Hausdorff measure must converge, $H_i$ must eventually be trivial.
\end{proof}

\begin{lemma}\label{l:ambientconvergence}
	Let $G$ be a compact Lie group and let $(X_i,G)$ be a sequence of $G$--spaces in $\meq$ which converges to $(X,\Gamma)$ in the equivariant Gromov--Hausdorff topology. Suppose further that this sequence of actions is equicontinuous. Then there is a subsequence $X_{i^j}$ such that there is a metric on $\mathcal{X} = X \coprod_j X_{i^j}$ that
	\begin{enumerate}
		\item restricts to the original metric on each of $X_{i^j}$ and $X$;
		\item is invariant with respect to an action of $G$, which restricts to the original action on each of the $X_{i^j}$; and
		\item induces a convergence of $X_{i^j}$ to $X$ in the Hausdorff metric on the closed subsets of $\mathcal{X}$;
	\end{enumerate} 
	and therefore $G$ is, after factoring out any ineffective kernel of its action on $X$, isomorphic to $\Gamma$.
\end{lemma}

\begin{proof}
	
	Fix Gromov--Hausdorff $\epsilon_i$--aproximations $f_i \co X_i \to X$ which witness the Gromov--Hausdorff convergence of the underlying metric spaces. Using these approximations, it is possible to define a limiting $G$--action on $X$ as follows.
	
	Consider the actions as continuous maps $\phi_i \co G \times X_i \to X_i$. Fixing a metric on $G$, the functions $\id \times f_i$ are Gromov--Hausdorff approximations showing the convergence of $G \times X_i$ to $G \times X$.  By the Grove--Petersen--Arzel\`{a}--Ascoli Theorem, one can extract from the equicontinuous subsequence $\phi_i$ a compact subsequence converging to a continuous map $\phi \co G \times X \to X$  \cite[Appendix]{gpjdg}. It is clear that this map is also an isometric action.
	
	Now pick approximations $g_i \co X \to X_i$ such that $g_i \circ f_i$ is close to the identity. Let $h_i \co X_i \to X_{i+1}$ be defined by $h_i = g_{i+1} \circ f_i$. Then $h_i$ is a Gromov--Hausdorff $5\epsilon_i$--approximation which is almost equivariant with respect to the action of $G$, and so $(h_i, \id, \id)$ can be used as an equivariant Gromov--Hausdorff $r_i$--aproximation. The quantity $r_i$ depends both on $\epsilon_i$ and on the rate of convergence of the $\phi_i$ to $\phi$.
	
	It is then possible to place a metric on the disjoint union $X_i \coprod X_{i+1}$ such that $\dist (x, h_i (x)) = r_i$ (see Burago, Burago and Ivanov \cite[Corollary 7.3.28]{BBI}).
	This metric can be rendered $G$-invariant by the usual averaging procedure, at a small cost---$h_i$ is now a Gromov--Hausdorff $3r_i$--approximation. The restriction of the metric to $X_i$ and to $X_{i+1}$ is unchanged.
	Let $d_i$ be the Hausdorff distance between $X_i$ and $X_{i+1}$ in this metric.
	
	Following Petersen \cite[p297]{petersen}, pass to a subsequence so that $d_i < 2^i$ for all $i$. Then, by gluing the metrics on each of the $X_i \coprod X_{i+1}$, a $G$--invariant metric on $\coprod_i X_i$ can be constructed, which restricts to the original metric on each $X_i$. This space can be completed to $\mathcal{X}$ in such a way that $\mathcal{X} = X \coprod_i X_i$, and $X_i$ converges to $X$ in the Hausdorff sense in $\mathcal{X}$.
	
	Since $\coprod_i X_i$ is dense in $\mathcal{X}$, the isometric $G$--action can be extended to an isometric action on all of $\mathcal{X}$, and the extension to $X$ is the limiting $G$-action constructed at the beginning of the argument.
	
	This action is, after factoring out any ineffective kernel, the limit of the $G$--actions in the equivariant Gromov--Hausdorff topology.
\end{proof}

It is now possible to proceed to the proof of Theorem A.
The proof relies on a result of the author from the general theory of transformation groups \cite{hcloseactions}, given here as Theorem \ref{t:cst}.
This result, along with the justification for its application here, is reviewed at the end of the paper, in Section \ref{s:tameness}.

\begin{proof}[Proof of Theorem A]
	
	\textbf{Envelop the convergence.}
	
	By Lemma \ref{l:ambientconvergence}, one may assume by passing to a subsequence that there is a $G$-invariant metric on $\mathcal{X} = X \coprod_i X_i$ which restricts to the original metrics and actions on each of the $X_i$, with $X_i$ converging to $X$ in the Hausdorff metric on the closed subsets of $\mathcal{X}$. Fix approximations $\theta_i \co X_i \to X$.
	
	Let $G'$ be the ineffective kernel of the $G$--action on $X$ (it will be shown later that this is trivial). Now $(X_i,G)$, converges to $(X,G/G')$ in the equivariant Gromov--Hausdorff topology. Let $\pi \co G \to G/G'$ be the projection map. Then equivariant Gromov--Hausdorff approximations from $(X_i, G)$ to $(X,\Gamma)$ which witness the convergence are given by the triple $(\theta_i, \pi, s)$ where $s$ is a (possibly discontinuous) section of $\pi$.
	
	\textbf{The cohomogeneity is constant.}
	
	By applying the slice theorem to $\mathcal{X}$, as a sequence of points $p_i \in X_i$ converges to $p \in X$ the isotropy group $G_p$ must be larger than $G_{p_i}$. In particular, the principal orbits of $X$ have dimension no greater than those of $X_i$. In other words, $\dim (X / G) \geq \dim (X_i / G)$.
	
	On the other hand, the orbit spaces $X_i/G$ converge to $X/G$ under a uniform lower curvature bound, so it follows that $X_i / G$ and $X / G$ have the same dimension, and are therefore homeomorphic by Perelman's Stability Theorem.
	
	\textbf{The radius of the tubes is bounded.}
	
	Here the term \emph{tube} is used in the transformation groups sense, meaning the image of a slice under translation by the group.
	
	Let $p \in X$, and let $\bar{p}$ be its image in $X/G$. As described by the author and Searle \cite[section 3.4]{HS}, a tube in an Alexandrov space around the orbit $G \cdot p$ can be constructed by choosing a strictly concave function $\bar{h}$ on a neighborhood $U$ of $\bar{p}$ which achieves its maximum at $\bar{p}$. This construction is due to Perelman \cite{permorse}.
	
	The gradient flow of $\bar{h}$ gives a retraction $\bar{r} \co U \to \bar{p}$.
	The function $\bar{h}$ lifts to a function $h$ on a neighborhood of $G \cdot p$.
	The gradient flow of $h$ then gives a $G$--invariant retraction $r$ onto $G \cdot p$, showing that neighborhood to be a tube around the orbit.
	
	By Perelman and Petrunin \cite[Lemma 4.3]{perpet}, the construction of $\bar{h}$ is such that strictly concave functions $\bar{h}_i$ exist on neighborhoods in $X_i / G$ converging to $h$.
	Let $\bar{p}_i$ be the maxima of the $\bar{h}_i$.
	Let $r$ be such that, for large $i$, $B_r(\bar{p}_i)$ is contained in the domain of concavity of $\bar{h}_i$.
	
	This establishes the existence of a sequence of points $p_i \to p$ such that there are tubes of a fixed radius $r$ around each $G \cdot p_i$ and $G \cdot p$. Clearly the orbits $G \cdot p_i$ are of the most singular type possible in the neighborhood.
	
	\textbf{The orbit-type survives passage to the limit.}
	
	It is claimed that for every subgroup $H \subset G$, after passing to a subsequence $X_i^H \to X^H$. 
	Recall that $X_i^H$ is a subset of $X_i/G$, the closure of the set of orbits with isotropy type $H$.
	
	Let $\bar{p}_i \in X_i^H$. Then there are points $p_i \in X_i$ above $\bar{p}_i$ which have isotropy containing $H$. Any accumulation point $p$ of the sequence $p_i$ is also fixed by $H$, and lies above some accumulation point of the $\bar{p}_i$.
	
	Next it is claimed that if, in fact, $p$ is fixed by some larger group $K$, then there is a sequence $q_i \to p$ of points in $X_i$ which are fixed by $K$.
	
	Fix $r$ so that the tube of radius $r$ around $G \cdot p$ can be approximated by tubes of radius $r$ about $G \cdot p_i$, with $p_i \to p$. 
	By Proposition \ref{p:localstab}, for large $i$ the tubes around $G \cdot p_i$ are homeomorphic to those around $G \cdot p$.
	Homeomorphism of the tubes implies homotopy equivalence of the orbits, so the orbits are all of the same dimension, and have the same number of components.
	In the case that $G$ is finite, this proves the claim.
	
	Consider a tube in the enveloping space $\mathcal{X}$ around $G \cdot p$, and fix for the remainder of the proof a decomposition of the tube into slices at each point of the orbit.
	After picking $p_i$ to lie in a slice at $p$, $G_{p_i}=L_i$ must be a subgroup of full dimension in $K$.
	
	Let $K_0$ be the identity component of $K$, and hence also of the $L_i$.
	Let $\Gamma = K/K_0$.
	Now it is clear that $X_i / K_0 \to X / K_0$, and this convergence is equivariant with respect to the action of $\Gamma$.
	
	Since $\Gamma$ fixes the image of $p$, $\bar{p} \in X/K_0$, there are points $\bar{q}_i \in X_i / K_0$ converging to $\bar{p}$ which are also fixed by $\Gamma$.
	As noted in the previous section of the proof, these points $\bar{q}_i$ are of the most singular type possible locally, and so they must correspond to fixed points of $K_0$, $q_i \in X_i$. 
	These $q_i$ then have isotropy type $K$ as required.
	\newpage
	\textbf{Construct the homeomorphisms.}
	
	Recall that the sets $X_i^H$ and $X^H$ are extremal subsets of the orbit space. 
	By the Stability Theorem, the convergence $X_i / G \to X/G$ inside $\mathcal{X}/G$ can then be used to establish homeomorphisms $\theta_i \co X/G \to X_i / G$ which carry $X^H$ to $X_i^H$ for every subgroup $H$.
	These $\theta_i$ are Hausdorff approximations in the space $\mathcal{X}/G$.
	
	Now consider the space $X/G$ as an abstract orbit space (see Definition \ref{d:abstract}). Let $f \co X/G \to \mathcal{X}/G$ be the obvious embedding of $X/G$ as the orbit space of $X \subset \mathcal{X}$.
	Let $f_i$ be the embedding $\theta_i \circ f$. Clearly $f_i$ converges to $f$.
	Now each of the $X_i$ is a $G$--space over $X/G$. 
	The convergent sequence of embeddings into $\mathcal{X}/G$ can be used to apply the Covering Sequence Theorem \ref{t:cst}, to obtain strong equivalence of the $X_i$, ie, equivariant homeomorphisms of $X_i$ with $X$ which descend to $\theta_i$.
	
	\textbf{Remove the subsequence.}
	
	Return to consideration of the original equicontinuous sequence $(X_i, G)$. If there is no $N_0$ such that, for all $n \geq N_0$, the space $(X_n, G)$ is equivariantly homeomorphic to the limit $(X,G)$, there would be a subsequence $(X_{i^j},G)$ of spaces converging to $(X,G)$, but none of which are equivariantly homeomorphic to $(X,G)$. However, by what has already been shown, that subsequence must itself have a subsubsequence which in fact is equivariantly homeomorphic to $(X,G)$ in the tail, and this would yield a contradiction.
\end{proof}

These arguments can easily be applied to pointed convergence of non-compact spaces in the case where the group fixes the basepoint.

\begin{corollary}Let $G$ be a compact Lie group and let $(X_i,p_i)$ be a sequence of pointed complete Alexandrov spaces of dimension $n$ and curvature bounded below by $k$. Let $G$ act isometrically on each of $X_i$, fixing $p_i$. Suppose the sequence converges to an action of $\Gamma$ on another $n$--dimensional complete pointed Alexandrov space $(X,p)$ in the equivariant Gromov--Hausdorff topology. Suppose further that, for every $r > 0$, this sequence of actions is equicontinuous on $B_r(p_i)$. 
	
Then for each $R,\epsilon>0$ and for large $i$ there are open equivariant embeddings $$\psi_i \co B_{R+\frac{\epsilon}{2}}(p) \to B_{R+\epsilon}(p)$$ covering $B_R(p_i)$. Furthermore, there is a subsequence such that the embeddings can be chosen to cover stability embeddings of the orbit spaces.\end{corollary}

The non-equivariant version of the stability theorem may be rephrased as follows: For every $X$ in the class of Alexandrov spaces of dimension $n$ with curvature bounded below by $k$, there is an $\epsilon = \epsilon (X,k)$ such that every space in the class within Gromov--Hausdorff distance $\epsilon$ of $X$ is homeomorphic to $X$.

Theorem A can be rephrased in the same manner, using Proposition \ref{p:equicontinuity}.

\begin{theorem}
Let $G$ be a compact Lie group acting by isometries on a compact Alexandrov space $X$ of dimension $n$ and curvature bounded below by $k$. Then there is some $\epsilon = \epsilon (X,G,k)$ such that any compact Alexandrov space of dimension $n$ and curvature bounded below by $k$ with an isometric $G$--action which is within equivariant Gromov--Hausdorff distance $\epsilon$ of $(X,G)$ is equivariantly homeomorphic to $(X,G)$.
\end{theorem}

\FloatBarrier

\section{Orbifolds}\label{ss:orbifolds}

Orbifolds were first introduced by Satake under the name \emph{V--manifolds} \cite{Satake}, as topological spaces locally modelled on a quotient of Euclidean space by a finite group. Some basic facts about orbifolds are reviewed here. The reader may refer to, among others, the book by Adem, Leida and Ruan \cite{alr} or Thurston's notes \cite{thurston} for further information.

\begin{definition}A smooth $n$--dimensional \emph{orbifold chart} over a topological space $U$ is a triple $(\tilde{U}, \Gamma_U, \pi_U)$ such that $\tilde{U}$ is a connected open subset of $\rrr^n$, $\Gamma_U$ is a finite group of smooth automorphisms of $\tilde{U}$ and $\pi_U \co \tilde{U} \to U$ is a $\Gamma_U$--invariant map inducing a homeomorphism $\tilde{U}/\Gamma_U \cong U$.\end{definition}

For convenience, a chart will sometimes be referred to as being over a point $p$. 
This will mean that the chart is over some neighborhood of $p$.

Let $U$ and $V$ be open subsets of a topological space $X$, and let $(\tilde{U}, \Gamma_U, \pi_U)$ and  $(\tilde{V}, \Gamma_V, \pi_V)$ be orbifold charts of dimension $n$ over $U$ and $V$ respectively.
The charts are called \emph{compatible} if, for every $p \in U \cap V$, there is a neighborhood $W$ of $p$ and an orbifold chart $(\tilde{W}, \Gamma_W, \pi_W)$ over $W$ such that there are smooth embeddings $\lambda_U \co \tilde{W} \embed \tilde{U}$ and $\lambda_V \co \tilde{W} \embed \tilde{V}$ with $\pi_V \circ \lambda_V = \pi_W$ and $\pi_U \circ \lambda_U = \pi_W$.

As usual, an orbifold atlas on a space $X$ will mean a collection of compatible charts covering $X$. Now the definition of an orbifold can be made.

\begin{definition}A smooth \emph{orbifold} of dimension $n$ is a paracompact Hausdorff space equipped with an atlas of orbifold charts of dimension $n$.\end{definition}

An orbifold homeomorphism (respectively diffeomorphism) is a homeomorphism of the underlying topological space which can locally be lifted to an equivariant homeomorphism (respectively diffeomorphism) of charts. Borzellino and Brunsden (see \cite{borzbrunsden}) have pointed out that this definition of an orbifold map is not sufficient for many purposes, though it is appropriate to the question currently under consideration. Four possible definitions of orbifold map are given there, of which this is the most na\"{i}ve notion, the \emph{reduced orbifold map}.

Let $X$ be an orbifold, let $p \in X$, and let $(\tilde{U}, \Gamma, \pi)$ be a chart over $p$ with $\pi(y) = p$.
The isotropy group of $y$ will be called the \emph{local group} at $p$, and will be written as $\Gamma_p$. 
It is uniquely defined up to conjugacy in $\Gamma$, and choosing a different chart does not change the isomorphism type of the group. 

In fact, one can always choose a linear chart over $p$ such that the group of automorphisms is isomorphic to $\Gamma_p$. By this is meant a chart of the form $(\rrr^n, \Gamma_p, \pi)$ where the action of $\Gamma_p$ is via a faithful orthogonal representation $\rho_p \co \Gamma_p \embed \on$. Such a chart will be referred to as a \emph{linear chart around $p$}. The representation is also uniquely determined up to isomorphism, and will be called the \emph{local action} at $p$. The differential of the action of $\Gamma_p$ at the origin of the chart is also isomorphic to $\rho_p$.

A \emph{Riemannian metric} on an orbifold can be given by fixing a finite atlas and a partition of unity with respect to the corresponding cover, and choosing Riemannian metrics on the charts which are invariant with respect to the finite group action. An orbifold equipped with a Riemannian metric is called a Riemannian orbifold. Once the metric on the orbifold is given it can be lifted to the maximal atlas in a canonical manner. The various notions of curvature at points of an orbifold can then be defined by reference to the curvature of the charts.

It is straightforward to see that an orbifold with sectional curvature $\geq k$ is also an Alexandrov space with curvature $\geq k$.
The tangent cone at any point of an orbifold is then well-defined, and coincides with the usual notion of tangent space for orbifolds.
The notion of an extremal set now finds a very natural application in orbifolds.

\begin{proposition}\label{p:extremalstrata}Let $X$ be an orbifold of dimension $n$, $\Gamma$ a finite group, and $\rho \co \Gamma \embed \on$ a linear representation of $\Gamma$. Let $X^{\rho}$ be the closure of all points with local action $\rho$. Then $X^{\rho}$ is an extremal set of $X$.\end{proposition}

\begin{proof}The result is clear where $n = 1$. Let $p \in X^{\rho}$ and consider the local action at $p$ by $\Gamma_p$. The tangent cone at $p$ is the cone on the quotient of the unit sphere by $\Gamma_p$. Consider the image of those points in the unit sphere having isotropy isomorphic to $\rho$. The closure of the cone on this set is $T_p X^{\rho}$ and by induction it is extremal in $T_p X$.  Since $X^{\rho}$ is closed, it is extremal.\end{proof}

The following lemma now shows that a linear chart around $p$ can be extended over any cone-like set around $p$. 

\begin{lemma}\label{l:chart}Let $X$ be an orbifold, and let $p \in X$. Let $U$ be a cone-like set around $p$. Then there is a linear chart over $U$ around $p$.\end{lemma}
\begin{proof}
	Consider the differential of the local action of $\Gamma_p$ on $\rrr_n$.
	The quotient of this action is the tangent cone at $p$, $T_p X$.
	
	Let $f \co U \to T_p X$ be a homeomorphism carrying each extremal set $E$ in $U$ to $T_p E$.
	Note that because $U$ is cone-like, $f$ preserves the local action at every point.
	
	Using a maximal atlas, cover $U$ by the ranges of all possible linear charts, $\{U_{\kappa}\}_{\kappa \in K}$. 
	Discard any $U_{\kappa}$ such that $f(U_{\kappa})$ is not the range of a linear chart in $T_p X$.
	
	Observe that this reduced family still covers $U$. 
	Suppose some $q \in U$ is not in any element of the reduced family. 
	Then for every $\kappa \in K$ such that $q \in U_{\kappa}$, $f(U_{\kappa})$ is not a linear chart.
	But $f(q)$ is covered by \emph{some} linear chart, and the intersection $W$ of the range  of this chart with $f(U_{\kappa})$ is also covered by a linear chart.
	Then because $f^{-1}(W) \subset U_{\kappa}$ it too is covered by a linear chart.
	It follows that $f^{-1}(W)=U_{\lambda}$ for some $\lambda \in K$, and is in the reduced family.
	
	Select a countable subcover, $U_1, U_2, \ldots$, and write $V_i$ for $f(U_i)$.
	Let $\Gamma_i$ be the local group acting on the charts $U_i$ and $V_i$.
	The charts $\tilde{V}_1, \tilde{V}_2, \ldots$ can be glued together to construct a chart over all of $T_p X$.
	The gluing requires $[\Gamma_p : N(\Gamma_i)]$ copies of $\tilde{V}_i$.
	The manner of this gluing gives a set of instructions which allows one to glue the charts $\tilde{U}_1, \tilde{U}_2, \ldots$ together to obtain the desired chart $\tilde{U}$.
	
	Since this chart is built by gluing together charts from the orbifold atlas, it is compatible with the atlas.
\end{proof}

By the Stability Theorem \ref{t:relstability}, $\o$ contains only finitely many topological types. To prove Theorem B, it is therefore sufficient to prove the following.

\begin{theorem}Let $X$ be a compact topological space. Then, up to orbifold homeomorphism, there are only finitely many orbifold structures on $X$ which belong to $\o$.\end{theorem}

\begin{proof}
	Aiming for a contradiction, let $O_i$ be a sequence of orbifolds in $\o$, all of which have underlying topological space $X$, and no two of which are orbifold homeomorphic. 
	By compactness of $\a$, a subsequence of $O_i$ converges in the Gromov--Haudorff sense to some $Y \in \a$ which also has underlying space $X$. 
	Abusing notation, the subsequence will still be written as $O_i$.
	Since there will be many more instances of passing to subsequences, this abuse of notation will be repeated throughout the proof.
	
	By Stanhope \cite{stan} there is a uniform upper bound on the order of the local group of a point in $\o$. 
	Recall a finite group has only finitely many linear representations in a given dimension.
	It follows that all the possible local actions up to isomorphism can be listed by $\rho_{\ell} \co G_{\ell} \to \mathrm{GL}(n)$, for ${\ell} = 1 , \ldots , m$ where $m$ is some finite number.
	Let $E_i^{\ell}$ be $O_i^{\rho_{\ell}}$, the closure of the subset of $O_i$ with local group action isomorphic to $\rho^{\ell}$.
	By Proposition \ref{p:extremalstrata} the $E_i^{\ell}$ are extremal sets.

	Passing to a subsequence $m$ times if necessary, one may assume that each sequence $E_i^{\ell}$ converges to an extremal subset $E^{\ell} \subset Y$. 
	Now, by the Stability Theorem \ref{t:relstability}, there are homeomorphisms $h_i \co Y \to O_i$ which are Gromov--Hausdorff approximations and carry each of the $E^{\ell}$ onto the $E_i^{\ell}$. 
	
	To prove the result, it is now sufficient to show that $h_{ij} \co O_i \to O_j$ given by $h_{ij} = h_j \circ h_i^{-1}$ is an orbifold homeomorphism.
	
	Let $\pal$ be a set of points in $Y$ such that $Y$ is covered by cone-like metric balls $\ual$ centered at $\pal$. 
	Then the sets $h_i (\ual)$ are also cone-like around $\pali = h_i (\pal)$, and cover $O_i$. 
	Denote these sets by $\uali$.
	
	By Lemma \ref{l:chart} each $\uali$ is covered by a chart $(\cali,\Gamma_{\pali},\pi_{\uali})$. 
	By passing to a subsequence, we may assume that the $\cali$ form a convergent sequence in the pointed equivariant Gromov--Hausdorff topology, converging to some object $(\cal,\Gamma_{\pal}) \in \mathcal{M}^c_{eq}$.
	
	Now, by Theorem A, $\cali$ and $\cal$ are equivariantly homeomorphic by some $F_i \co \cal \to \cali$. 
	The $F_i$ induce homeomorphisms  $f_i \co \cali / \Gamma_{\pali} \to \cal / \Gamma_{\pal}$ which are Hausdorff approximations witnessing the Hausdorff convergence of the orbit spaces inside the enveloping orbit space. 
	
	Write $\mu\ali$ for the isometry $\cali / \Gamma_{\pali} \to \uali$ induced by $\pi_{\uali}$.
	
	\begin{center}
		\begin{tikzcd}
			\cali 					\arrow[two heads]{d}	& \cal \arrow{l}{F_i}[swap]{\cong} \arrow[two heads]{d} \\
			\cali / \Gamma_{\pali}	\arrow{d}{\mu\ali}[anchor=center,rotate=-90,yshift=-1ex,xshift=-0.2ex]{\cong}	& \cal / \Gamma_{\pal} \arrow{l}{f_i}[swap]{\cong}\\
			\uali											& \ual \arrow{l}{h_i}[swap]{\cong}
		\end{tikzcd}
	\end{center}
	
	Now the gap may be filled in by a homeomorphism $\phi_i \co \cal / \Gamma_{\pal} \to \ual$ given by $h_i ^{-1} \circ \mu\ali \circ f_i$. The $\phi_i$ make up a sequence of Gromov--Hausdorff approximations, and the sequence converges to some isometry $\phi \co \cal / \Gamma_{\pal} \to \ual$. 
	
	Then the $f_i$ may be adjusted slightly, setting $g_i = (\mu\ali)^{-1} \circ h_i \circ \phi$.
	Since $\phi_i$ converges to $\phi$, these homeomorphisms $g_i$ will also witness the Gromov--Hausdorff convergence of $\cali / \Gamma_{\pali}$ to $\cal / \Gamma_{\pal}$.
	By Theorem A, new equivariant homeomorphisms $G_i \co \cal \to \cali$ can be chosen which will induce the homeomorphisms $g_i$.
	
	This gives a non-smooth orbifold chart over $\ual$, $(\cal, \Gamma_{\pal}, \phi)$ such that the $h_i \co \ual \to \uali$ are orbifold homeomorphisms. The maps $h_{ij}$ are then also orbifold homeomorphisms.
\end{proof}

\section{Tameness of Alexandrov spaces}\label{s:tameness}

This section will provide the necessary background to Theorem \ref{t:cst} and justify its application in Theorem A.
The result is a refinement of Palais' classification of $G$--spaces \cite{palais} for orbit spaces which are ``tamely partitioned", and appears in \cite{hcloseactions}

For a subgroup $H$ of $G$, write $(H)$ for the conjugacy class of $H$. Say that $(H) \leq (K)$ if $K$ has a subgroup which is conjugate to $H$.

\begin{definition}\label{d:abstract}
	Let $G$ be a compact Lie group. Then an \emph{abstract orbit space} for $G$ is a locally compact, second countable space $Z$ together with a partition $\{ Z_{(H)}\}_{H \subset G}$ of $Z$ such that, for each $(H)$, $\cup \left\lbrace Z_{(K)} \st (K) \leq (H) \right\rbrace $ is open.
\end{definition}

A $G$--space over $Z$ is then a space with an action of $G$ by homeomorphisms, such that $X/G$ is homeomorphic to $Z$, via a homeomorphism that carries the orbit-type partition of $X/G$ to the partition on $Z$.

The notion of tameness used is quite a mild topological property, and it will be shown that the orbit spaces of isometric group actions on Alexandrov spaces satisfy it.
The definition first requires the concept of a controlled homotopy.

\begin{definition}
	Let $\mathcal{U}$ be an open cover of a topological space $Z$. Then a map $f \co Y \times [0,1] \to Z$ is called a \emph{$\mathcal{U}$--homotopy} if for each $y \in Y$ there is some $U \in \mathcal{U}$ such that $f(y,t) \in U$ for all $0 \leq t \leq 1$.
	
	In the case where $Z$ is metric and $\mathcal{U}$ is a cover by metric balls of radius $\frac{\epsilon}{2}$, a $\mathcal{U}$--homotopy may be called an \emph{$\epsilon$--homotopy}.
\end{definition}

It will be convenient to move from the partition on $Z$ to the related filtration.
For the purposes of this paper, a filtration may have an index set which is only partially ordered.
For $(H) \in \ooo$, write $Z_{\geq (H)}$ for the union of all $Z_{(K)}$ such that $(K) \geq (H)$. 
$Z_{\geq (H)}$ is a closed set.
These sets $Z_{\geq (H)}$ make up a filtration of $Z$ indexed by the partially ordered set $\ooo$, but with the reverse ordering.

\begin{definition}
	Let $Z$ be a filtered set (with the filtration indexed by a set which is possibly only partially ordered).
	The filtration is said to be \emph{tame} if for each $Y \subset Z$ which is a union of elements of the filtration
	and for each open cover $\mathcal{U}$ of $Z$
	there are a neighborhood $V$ of $Y$ 
	and a homotopy $h \co (Z \setminus Y) \times I \to Z \setminus Y$ satisfying:
	\begin{enumerate}
		\item $h$ is the identity on $(Z \setminus Y) \times \left\lbrace 0 \right\rbrace$,
		\item $h((Z \setminus Y) \times \left\lbrace 1\right\rbrace ) \subset Z \setminus V$,
		\item $h$ is a $\mathcal{U}$--homotopy, and 
		\item $h$ preserves every member of the filtration on $Z \setminus Y$. 
	\end{enumerate}
\end{definition}

The key result on $G$--spaces which is applicable in the proof of Theorem A can now be stated.

\begin{theorem}[Covering Sequence Theorem \cite{hcloseactions}]\label{t:cst}
	Let $X$ be a $G$--space having finitely many orbit-types, and let $Y = X/G$ be its orbit space. Let $Z$ be an abstract orbit space which is compact, metrizable, and tamely partitioned. Let $f_n \co  Z \to Y$ be a sequence of embeddings of $Z$ which carry the partition of $Z$ to the orbit-type partition of $Y$, restricted to the image of $f_n$. Suppose that $f = \lim_{n \to \infty} f_n$ exists, and is also such an embedding. 
	
	Then, for large enough $n$, the invariant subspaces of $X$ over the images of $f_n$ are equivariantly homeomorphic to that over the image of $f$, and the equivariant homeomorphisms induce the maps $f \circ f_n^{-1}$.
\end{theorem}

In the case under consideration, where the orbit space $Z$ is an Alexandrov space, and its partition is by extremal subsets, the remaining results of this section gives the necessary ``tameness'' requirement for application of the theorem. 

First, it will be established that tameness is a local property. 
Say that the filtration is locally tame if, for each closed $Y \subset Z$ as above, and for each $y \in Y$, there is an open set $U_y$ containing $y$ so that for each open cover $\mathcal{U}$ of $Z$ there is a filtration-preserving $\mathcal{U}$--homotopy $r \co (U_y \setminus Y) \times [0,1] \to Z\setminus Y$ deforming $U_y \setminus Y$ into $Z\setminus V$ for some open $V \supset Y$.

\begin{proposition}\label{p:localtameness}Let $Z$ be a compact metrizable space with a filtration. The filtration is tame if and only if it is locally tame.\end{proposition}

\begin{proof}
	Endow $Z$ with a metric, and replace the given cover $\mathcal{U}$ with a  cover by $\epsilon$--balls, and aim to construct an $\epsilon$--homotopy.
	
	Cover $Z$ with finitely many open sets $U_1, \ldots, U_N$ so that $Y$ is tame in each $U_i$.
	Choose continuous functions $a_i \co Z \to [0,1]$ so that the support of each $a_i$ is in $U_i$ and $\Sigma_i a_i = 1$.
	Let $r_i \co (U_i \setminus Y) \times [0,1] \to Z \setminus Y$ be an $\frac{\epsilon}{N}$--homotopy deforming $U_i \setminus Y$ into $Z \setminus V$ in a stratum-preserving manner for some open $V \supset Y$.
	(Since $N$ is finite, $V$ may be assumed not to depend on $i$.)
	By an appropriate choice of $r_i$, one may assume further that $r_i ((U_i  \setminus  Y) \times [\frac{1}{N+1} , 1]) \subset Z \setminus V$ and that $a_j \circ r_i$ is a $\frac{1}{(N+1)^3}$--homotopy for each $j$.  
	Extend each $r_i$ over $Z \times \left\lbrace 0\right\rbrace $ by the identity.
	
	New homotopies $R_i \co (Z  \setminus  Y) \times [0,1] \to Z \setminus Y$ can be constructed by $R_i(x,t) = r_i (x, a_i(x)t)$. Write $R^j \co (Z  \setminus  Y) \times [0,1] \to Z$ for the homotopy given by concatenating $R_1, \ldots R_j$. It is claimed that $R^N$ is the required deformation.
	
	Certainly since each $r_i$ is stratum-preserving, each $R_i$ is also, and so is each $R^j$. Since each $r_i$ is an $\frac{\epsilon}{N}$--homotopy, $R^N$ is an $\epsilon$--homotopy. 
	It remains only to show that $R^N((Z \setminus Y) \times \left\lbrace 1 \right\rbrace) \subset Z \setminus U$ for some open neighborhood $U$ of $Y$.
	
	For each $q \in Z \setminus Y$, there is some $i$ so that $a_i (q) \geq \frac{1}{N}$. 
	Because each of the homotopies $r_k$ changes the value of $a_i$ by no more than $\frac{1}{(N+1)^3}$, there is some $k \leq N$ so that $a_k (R^{k-1}(q,1)) > \frac{1}{N} - \frac{1}{(N+1)^2} > \frac{1}{N+1}$ and hence $R^k (q,1) \in Z \setminus V$. 
	In other words, every $q \in Z \setminus Y$ enters the compact subset $Z \setminus V$ at some point in the construction of $R^N$.
	The homotopy will continue to deform the subset $Z \setminus V$, but its image must remain compact.
	
	It follows that $R^N$ deforms $Z \setminus Y$ into a compact subset, and its complement is the desired $U$.
\end{proof}

\begin{proposition}
	Let $Z$ be a compact Alexandrov space, and let $E \subset Z$ be an extremal subset. Then for each $\epsilon > 0$ there are a neighborhood $V$ of $E$ and a homotopy $h \co (Z  \setminus  E) \times I \to Z  \setminus  E$ satisfying:
	\begin{enumerate}
		\item $h$ is the identity on $(Z  \setminus  E) \times \left\lbrace 0 \right\rbrace$, 
		\item $h((Z  \setminus  E) \times \left\lbrace 1\right\rbrace ) \subset Z  \setminus  V$,
		\item $h$ is an $\epsilon$--homotopy, and 
		\item $h$ preserves the extremal subsets of $Z  \setminus  E$. 
	\end{enumerate}
\end{proposition}

\begin{proof}
	The proof is by induction on the dimension of $Z$.
	If $\dim (Z) = 1$, $Z$ is a circle or a closed interval, and the result clearly holds.
	Suppose the result has been shown for all compact Alexandrov spaces of dimension at most $n-1$.
	
	By Proposition \ref{p:localtameness} it is sufficient to show the result locally.
	Cover $E$ by sets $U_1, \ldots, U_N$ which are cone-like around points $p_1, \ldots, p_N \in E$.
	The result is shown if it can be shown for each $U_i$.
	
	Choose a finite cover of $Z$ by balls of radius $\frac{\epsilon}{2N}$.
	Since each $U_i$ is cone-like, these give finite covers $\mathcal{U}_i$ of each $T_{p_i} Z$.
	An inspection of the proof of Proposition \ref{p:localtameness} shows that it will be sufficient to construct $\mathcal{U}_i$--homotopies on each $T_{p_i} Z$.
	
	Since the tangent cone is not compact, Proposition \ref{p:localtameness} does not apply directly to $T_{p_i} Z$ itself.
	However, it does apply to the space of directions, and so it is not too hard to adapt it to the cone.

	Let $o$ be the vertex of $T_{p_i} Z$.
	By the compactness of $\Sigma_{p_i} Z$ (which will be written $\Sigma_i$ for convenience), there are an unbounded increasing sequence of numbers $0<t_0 < t_1 < t_2, \ldots$ and finite covers $\mathcal{V}_j$ of $\Sigma_i$ by balls of radius $\delta_j$ for each $j=0,1,2,\ldots$  so that  
	$$\mathcal{W}_i=
	\left\lbrace B_{o}(t_1)\right\rbrace  \cup
	\left\lbrace \mathcal{V}_j \times (t_j,t_{j+2}) : j=0,1,2,\ldots \right\rbrace .$$
	is a star-refinement of $\mathcal{U}_i$.
	
	By the induction hypothesis, for each $j$ there is a homotopy on $\Sigma_i  \setminus  \Sigma_{p_i}E$ deforming away from an open neighborhood of $\Sigma_{p_i}E$.
	This may be taken to create a $\mathcal{W}_i$--homotopy on $(\Sigma_i  \setminus  \Sigma_{p_i}E) \times (t_j,t_{j+2})$.
	By gluing these together, a $\mathcal{U}_i$--homotopy on $(\Sigma_i  \setminus  \Sigma_{p_i}E) \times (t_0,\infty)$ can be constructed.
	
	The homotopy can then easily be extended by coning to a $\mathcal{U}_i$--homotopy on $T_{p_i}Z  \setminus  T_{p_i}E$, but it will not deform away from an open neighborhood of the origin of the cone.
	If it is followed with a sufficiently small radial strong deformation retraction away from the origin, however, it will satisfy all the necessary properties.
	
	This gives a $\mathcal{U}_i$--homotopy on $T_{p_i}Z$ as required. These are $\frac{\epsilon}{N}$ homotopy on $U_i$, and by Proposition \ref{p:localtameness} can be glued together to given an $\epsilon$--homotopy on $Z$.
\end{proof}

\FloatBarrier

\bibliographystyle{amsabbrv}
\bibliography{C:/Users/John/mybib}

\end{document}